\documentclass[reqno]{amsart}
\usepackage{amsfonts, amsmath, amsthm, amssymb, latexsym, graphicx}

\theoremstyle{plain}
\newtheorem{theorem}{Theorem}[section]
\newtheorem{corollary}[theorem]{Corollary}
\newtheorem{lemma}[theorem]{Lemma}
\newtheorem{proposition}[theorem]{Proposition}
\theoremstyle{definition}
\newtheorem{definition}[theorem]{Definition}
\theoremstyle{remark}
\newtheorem*{remark}{Remark}
\newtheorem*{remarks}{Remarks}

\numberwithin{equation}{section}

\title[Limit transition between hypergeometric functions]
{Limit transition between hypergeometric functions of type BC and type A}
\author{Margit R\"osler} 
\address{Institut f\"ur Mathematik, Universit\"at Paderborn, Warburger Str. 100, D-33098 Paderborn, Germany}
\email{roesler@math.upb.de}
\author{Tom Koornwinder} 
\address{Korteweg de Vries Institute, University of Amsterdam, P.O. Box 94248, 1090 CE Amsterdam, The Netherlands}
\email{T.H.Koornwinder@uva.nl}
\author{Michael Voit}
\address{Fakult\"at Mathematik, Technische Universit\"at Dortmund,
          Vogelpothsweg 87,
          D-44221 Dortmund, Germany}
\email{michael.voit@math.tu-dortmund.de}


\begin{document}
%\date{\today}

\begin{abstract}
Let $F_{BC}(\lambda,k;t)$ be the Heckman-Opdam hypergeometric
function of type BC with multiplicities $k=(k_1,k_2,k_3)$ and weighted
half sum
$\rho(k)$ of positive roots. We prove that  
$F_{BC}(\lambda+\rho(k),k;t)$ converges for $k_1+k_2\to\infty$ and
$k_1/k_2\to \infty$ to a
function of type A for $t\in\mathbb R^n$ and $\lambda\in\mathbb
C^n$. This limit is obtained from
a corresponding result for Jacobi polynomials of type BC, which is
proven for a slightly more general limit
behavior of the multiplicities, using an explicit representation of
Jacobi polynomials in terms of Jack polynomials.

Our limits include limit transitions for the spherical functions of
non-compact
Grassmann manifolds over one of the fields $\mathbb F= \mathbb R,
\mathbb C, \mathbb H$ when the rank is fixed
and the dimension tends to infinity. The limit functions turn out to be 
 exactly the spherical functions of the corresponding infinite dimensional
Grassmann manifold in the sense of Olshanski.
\end{abstract}

\maketitle

\section{Introduction}
Consider the  Heckman-Opdam hypergeometric
functions $F_{R}(\lambda,k;t)$  for the root systems 
$R=BC_n=\{\pm e_i, \pm 2e_{i}, \pm e_i \pm e_j, \, 1\leq i <j\leq n\}\,$
and $A_{n-1}=\{ \pm(e_i-e_j): 1\leq i<j\leq n\}\,$ with
multiplicities $k=(k_1,k_2,k_3)$ and $k=\kappa$ respectively as studied
e.g. in \cite{BO}, \cite{H1},
\cite{H2}, \cite{H3}, \cite{HS}, \cite{O1}, \cite{O2}.
Fix a positive subsystem $R_+$ in each case and denote by 
$\rho_{R}(k) = \frac{1}{2}\sum_{\alpha\in R_+} k_\alpha \alpha$ the
weighted half-sum of positive roots.
The Jacobi polynomials of type $BC_n$ are indexed by the cone of
dominant weights
\[ P_+ = \{(\lambda_1, \ldots \lambda_n) \in \mathbb Z_+^n: \lambda_1
\geq \ldots \geq \lambda_n\}\]
and can be written as
\[ P_\lambda^{BC}(k;t) = \frac{1}{c(\lambda+\rho_{BC}(k),k)}
F_{BC}(\lambda+\rho_{BC}(k), k;t) \]
where $c$ is the generalized $c$-function. 
The Jacobi polynomials of type $A_{n-1}$  are indexed
by the set $\pi(P_+)$, where $\pi$ denotes the orthogonal projection
of $\mathbb R^n$ onto
$\mathbb R^n_0 $. They can be written as monic Jack polynomials, 
\[ P_{\pi(\lambda)}^A(\kappa;t) = j_\lambda ^\kappa (e^t), \quad t\in
\mathbb R_0^n;\]
see Section 4 for the precise notation.

In this paper, we shall prove the following limit  for the
Jacobi polynomials of type $BC_n$:
\begin{equation}\label{limit-intro2} \lim_{\substack{ k_1+k_2 \to
\,\infty\\ k_1/k_2\,\to \,a}}
P_\lambda^{BC}(k;t) \,=\, 4^{|\lambda|}\cdot
j_\lambda^{k_3}(x(t))\end{equation}
for $a\in [0,\infty],$ with the transform
\[ \mathbb R^n\to \mathbb R_+^n, \quad t\mapsto x(t)
\quad\text{with}\quad
x_i(t)= \gamma_a+ {\rm sinh}^2\bigl(\frac{t_i}{2}\bigr), \quad
\gamma_a = \frac{a+1}{a+2}.\]
This result was already stated in \cite{K2} without proof.
A proof different from the one in the present paper was given by
R.~J. Beerends and the second author in an unpublished manuscript.

Restricting to  the  case $a=\infty$
we shall next extend the limit with respect to the spectral variable
$\lambda$ and prove that
\begin{align}\label{limit-intro1} \lim_{\substack{ k_1+k_2 \to \,\infty\\
                                k_1/k_2\,\to \,\infty}}
&F_{BC}(\lambda+\rho_{BC}(k),k;t) \\
= &\, \prod_{i=1}^n \bigl({\rm cosh}^2 \frac{t_i}{2}\bigr)^{\sum_{i=1}^n \lambda_i/n}\cdot 
 F_A\bigl(\pi(\lambda) + \rho_A(k_3), k_3; \pi\bigl(\log {\rm cosh}^2\frac{t}{2}\bigr)\bigr)
\notag\end{align}
for all  $t\in\mathbb R^n, $ locally uniformly in $\lambda\in\mathbb C^n$.

Let us briefly discuss the above limits for the rank one case $n=1$
where $k_3$ does not appear, the
functions $F_{BC}(\lambda,k;t)$ are essentially Jacobi functions
\[
\varphi_\lambda^{(\alpha,\beta)}(t)=
{}_2F_1(\tfrac12(\alpha+\beta+i\lambda),\tfrac12(\alpha+\beta+1-i\lambda);
\alpha+1;-\sinh^2t),
\]
for which we refer to \cite{K1}, and where $F_A$ reduces to the
constant function $1$.
More precisely, comparing the examples on p. 89 of \cite{O1} and
\cite{K1}, we have
$$F_{BC_1}(\lambda,k;t)=
\varphi_{-2i\lambda}^{(\alpha,\beta)}\bigl(\frac{t}{2}\bigr)
\quad\text{with}\quad  \alpha=k_1+k_2-1/2, \> \beta=k_2-1/2,$$
and (\ref{limit-intro1}) becomes
the limit
$$\lim_{\alpha\to\infty, \alpha/\beta\to\infty}
\varphi_{\lambda+i(\alpha+\beta+1)}^{(\alpha,\beta)}(t) 
= ({\rm cosh}\, t)^{-i\lambda} \quad(\lambda\in\mathbb C).$$
This limit is easily seen from
\begin{multline*}
\varphi_{\lambda+i(\alpha+\beta+1)}^{(\alpha,\beta)}(t)
={}_2F_1(\tfrac12i\lambda,\alpha+\beta+1-\tfrac12 i\lambda;
\alpha+1;-\sinh^2t)\\
=(\cosh t)^{-i\lambda}
{}_2F_1(\tfrac12 i\lambda,-\beta+\tfrac12 i\lambda;\alpha+1;\tanh^2t).
\end{multline*}
Moreover,  the Heckman-Opdam polynomials in rank one are related to 
the monic Jacobi polynomials $p_n^{(\alpha,\beta)}$ by
\[ P_n^{BC}(k;it) = 2^n p_n^{(\alpha,\beta)}(\cos t), \,\, n\in \mathbb Z_+.\]
Limit \eqref{limit-intro2} means that for $c = a+1\, \in [1, \infty)$
and $x=\cos t\in [-1,1]$,
$$\lim_{\alpha\to\infty, \alpha/\beta\to c}p_n^{(\alpha,\beta)}(x)
= \left(x+\frac{c-1}{c+1}\right)^n. $$
This limit is easily seen from
\[
p_n^{(\alpha,\beta)}(x)
=\frac{2^n(\alpha+1)_n}{(n+\alpha+\beta+1)_n}\,
\sum_{l=0}^n\frac{(n+\alpha+\beta+1)_l}{(\alpha+1)_l}\binom nl
\left(\frac{x-1}2\right)^l.
\]

We shall obtain \eqref{limit-intro2} by means of an explicit
representation of
$P_\lambda^{BC}$ in terms of Jack polynomials which goes back to ideas
of \cite{SK} and to \cite{Ha}.
The limit
\eqref{limit-intro1} for the hypergeometric function
is then  obtained from (\ref{limit-intro2}) 
by  Phragm\'{en}-Lindel\"of principles and
sharp explicit estimates for
general hypergeometric
functions which slightly improve estimates by Opdam \cite{O1} and
Schapira \cite{S}.
Our  limit transition (\ref{limit-intro1}) includes
a  limit result  for the spherical functions of the Grassmannians
$SO_0(p,n)/SO(p)\times SO(n),$   $SU(p,n)/S(U(p)\times U(n))$ 
and $Sp(p,n)/Sp(p)\times Sp(n), $ where $Sp(p,n)$ 
denotes the pseudo-unitary group of index $(p,n)$ over $\mathbb H$. As
$p\to\infty$
(and the rank $n$ is fixed),
the spherical functions of these Grassmannians converge to
(restrictions of) the spherical functions of the
reductive symmetric space $GL_+(n,\mathbb R)/SO(n), GL(n,\mathbb C)/U(n)$ and
$GL(n,\mathbb H)/Sp(n),$ respectively. We shall also show that the
obtained limits are exactly the spherical functions
of the corresponding infinite dimensional Grassmannians in the sense
of Olshanski.
Our results for infinite dimensional Grassmannians are also of
interest in comparison with
the recent results of \cite{DOW}. There 
it is shown that under natural conditions on an infinite dimensional
symmetric space
$\displaystyle G_\infty/K_\infty= \lim_{\rightarrow} G_n/K_n$ where
$G_n/K_n$ are Riemannian
symmetric of compact type, spherical functions of $G_n/K_n$ can have a
limit which is $K_\infty$-spherical only
if the $G_n/K_n$ are Grassmannians.

This paper is organized as follows: In Section 2 we recapitulate some
basic notions and facts on
the Cherednik kernel and Heckman-Opdam hypergeometric functions.
We need the Cherednik kernel because we improve in
Section 3 estimates of Opdam \cite{O1} and
Schapira \cite{S} for this function. This results in an estimate for the
Heckman-Opdam hypergeometric functions which is
uniform in the multiplicity parameters.
The Cherednik kernel will not be further used in the main part of the paper,
starting in Section 4, where
the limit (\ref{limit-intro2}) for Jacobi polynomials of type BC is proved. 
This result, the estimates of Section 3, and Phragm\'{en}-Lindel\"of
principles are combined in Section 5, leading to the limit
(\ref{limit-intro1}).
In Section 6 we briefly discuss this limit in terms of spherical
functions for non-compact Grassmann manifolds of growing dimension and
fixed rank.
Finally, in Section 7 the Olshanski spherical functions of the
associated infinite dimensional Grassmannians are characterized.

%%%%%%%%%%%%%%%%%%%%

\section{Notation and Preliminaries}

Let $\mathfrak a$ be a finite-dimensional Euclidean space with inner
product $\langle\,.\,,.\,\rangle$
which is extended to a complex bilinear form on the complexification
$\mathfrak a_\mathbb C$ of $\mathfrak a$.
We identify $\mathfrak a$ with its dual space
$\mathfrak a^* =\text{Hom}(\mathfrak a, \mathbb R)$ via the
given inner product.  Let $R\subset \mathfrak a$ be a (not necessarily reduced)
crystallographic
root system and let  $W$ be the Weyl group of $R$. For $\alpha\in R$ we write
$\alpha^\vee = 2\alpha/\langle\alpha,\alpha\rangle$ and denote by
$\sigma_\alpha(t) =  t - \langle t,\alpha^\vee\rangle \alpha\,$ the
orthogonal reflection
in the hyperplane perpendicular to $\alpha$.
We 
denote by $\mathcal K$ the  vector space of multiplicity functions
$k=(k_\alpha)_{\alpha\in R}$, satisfying $k_\alpha = k_\beta$ if $\alpha$ and
$\beta$ are in the same $W$-orbit. We shall write $k\geq 0$ ($k>0$) if
$k_\alpha\geq 0\,$ ($k_\alpha >0$) for all $\alpha\in R$.
 For $k\in \mathcal  K$ let 
\begin{equation}\label{general-halfsum}
\rho= \rho(k) := \frac{1}{2}\sum_{\alpha\in R_+} k_\alpha\alpha
\end{equation}
be the weighted half-sum of positive roots, where $R_+$ is some fixed positive
  subsystem of $R$. Let 
\[ \mathfrak a_+ := \{t\in \mathfrak a: \langle t, \alpha\rangle > 0 \,\,\forall \alpha\in R_+\}\]
be the positive Weyl chamber associated with $R_+.$ 
If $k\geq 0$, then $\rho(k)\in \overline{\mathfrak a_+},$ and if $k>0$, then  $\rho(k) \in \mathfrak a_+\,.$ This follows from the fact that for a simple system
$\{\alpha_i\}\subset R_+$ (with indivisible roots $\alpha_i$), the reflection $\sigma_{\alpha_i}$ leaves $R_+\setminus\{\alpha_i\}$ invariant, and hence
\[\langle \rho(k), \alpha_i^\vee\rangle = k_{\alpha_i} + 2k_{2\alpha_i}\]
(with the understanding that $k_{2\alpha_i} = 0\,$ if $2\alpha_i\notin R$),
c.f. \cite{M},
 Section 11. 

\noindent
For fixed $k\in \mathcal K,$ 
the Cherednik operator in direction $\xi\in \mathfrak a$ is  defined by
\[ T_\xi = T_\xi(k) :=\partial_\xi + \sum_{\alpha\in R_+} k_\alpha\langle\alpha,\xi\rangle \frac{1}{1-e^{-\alpha}}(1-\sigma_\alpha) - \langle\rho(k),\xi\rangle\]
where $\partial_\xi$ is the usual directional derivative  and  
\[e^\lambda(t) := e^{\langle\lambda,t\rangle}\quad \forall \lambda, t \in \mathfrak a_{\mathbb C}.\]
For fixed  $k$, 
the operators $\{T_\xi(k),\, \xi\in \mathfrak a\}$ commute. According to Theorem 3.15 of \cite{O1},
there exist a $W$-invariant tubular neighborhood $U$ 
of $\mathfrak a$ in $\mathfrak a_{\mathbb C}$ and a unique holomorphic function $G$ on 
 $\mathfrak a_{\mathbb C}\times K^{reg}\times U$ which satisfies 
\begin{align}\label{Cherednik_char}
{\rm (i) }\, & \,   \forall \,\xi\in \mathfrak a, \, \lambda\in \mathfrak a_{\mathbb C}: \,\, T_\xi(k)G(\lambda, k;\,.\,) = \langle\lambda,\xi\rangle G(\lambda, k;\,.\, );\notag \\
{\rm (ii) }\,&\,  G(\lambda, k; 0) = 1.   \end{align}
The function $G$ is called the Cherednik-Opdam kernel. We shall mainly be concerned with the hypergeometric function associated 
with $R$, which is given by
\[ F(\lambda,k;t) := \frac{1}{|W|} \sum_{w\in W} G(\lambda,k;w^{-1}t).\]
It is actually $W$-invariant
both in $\lambda$ and $t.$ The functions
$F(\lambda,k;\,.\,)$ generalize the spherical functions of Riemannian
symmetric spaces of the non-compact type, which occur for specific values of
the multiplicity parameter $k\ge 0.$ 

In order to interpret the main results below in the  geometric context, we shall use the following scaling property:

\begin{lemma}\label{scaling}
Let $R$ be a root system in a Euclidean space $\mathfrak a$ 
with multiplicity function  $k$.  For a constant $c>0$ consider the rescaled  root system
$ \widetilde R:= cR:= \{c\alpha, \alpha \in R\}$ and define $\widetilde k\,$ on $\widetilde R$ by 
$\,\widetilde k_{c\alpha} := k_\alpha\,.$
Then the associated Cherednik kernels are related via
\[G_\lambda(\widetilde k;t) \,= \, G_{\lambda/c}(k;ct).\]
A corresponding result holds also for the associated hypergeometric functions.
\end{lemma}

\begin{proof}
 Write  $\widetilde f(t) = f(ct)$ for functions $f$ on $\mathfrak a.$
Then 
\[ (T_{\xi}(\widetilde k) \widetilde f\,)(t) = (T_{c\xi}(k)f)(ct).\]
In  view of  characterization \eqref{Cherednik_char}, this implies the assertion.\end{proof}

\smallskip
In this paper, we shall always assume that $k\geq 0$ 
and we often write 
\[G(\lambda,k;t) = G_\lambda(k;t), \,\,\,F(\lambda,k;t) = F_\lambda(k;t).\]

\medskip

For certain spectral variables $\lambda$, the hypergeometric functions $F_\lambda$ are actually exponential polynomials, called Heckman-Opdam Jacobi polynomials.
To introduce these, let 
$P=\{ \lambda\in \mathfrak a: \langle\lambda,\alpha^\vee\rangle \in \mathbb Z \,\,\forall\,\alpha\in R\}\,$ denote the 
weight lattice of  $R$ and 
$\, P_+ =\{ \lambda\in P: \langle\lambda,\alpha^\vee\rangle \geq 0 \,\,\forall\,\alpha\in R_+\}\,$  the set of dominant weights associated with~$R_+$. We equip $P_+$ with the 
usual dominance order, that is, $\mu <\lambda$ iff $\lambda-\mu$ is a sum of positive roots. 
Let 
\[\mathcal T:=\text{span}_{\mathbb C}\{e^{\lambda}, \,\lambda\in P\}\]
denote the space of exponential polynomials associated with $R$.
The monomial symmetric functions
\begin{equation}\label{monomial} M_\lambda = \sum_{\mu\in W\lambda} e^\mu, \quad \lambda\in P_+\end{equation}
form a basis of the subspace $\mathcal T^W$ of $W$-invariant elements from $\mathcal T.$ 

%\medskip
%The  Jacobi polynomials $\{P_\lambda(k), \,\,\lambda\in P_+\}$ associated with
%$R$ are uniquely characterized by
%the following two conditions:
%\begin{enumerate}
%\item[\rm{(i)}] $\displaystyle P_\lambda(k) = M_\lambda + \sum_{\mu< \lambda}
%c_{\lambda\mu}(k)M_\mu\,\quad (c_{\lambda\mu}(k)\in \mathbb C);$
%\item[\rm{(ii)}] $\displaystyle L_k P_\lambda(k) = \langle \lambda, \lambda
%+2\rho(k)\rangle P_\lambda(k)$
%\end{enumerate}
%with the operator 
%\begin{equation}\label{Laplace} L_k = \Delta_{\mathfrak a} + \sum_{\alpha\in R_+}
%k_\alpha \,\text{coth}\frac{\langle \alpha, t\rangle}{2}\, \partial_\alpha\,.
%\end{equation}
%This is just the $W$-invariant part of the Heckman-Opdam Laplacian, which is
%
%Tom begin insert
\begin{definition}\label{JacobiPolDef}
The  Jacobi polynomials $\{P_\lambda(k), \,\,\lambda\in P_+\}$ associated with
$R$ are uniquely characterized by
the following two conditions:
\begin{enumerate}
\item[\rm{(i)}] $\displaystyle P_\lambda(k) = M_\lambda + \sum_{\mu< \lambda}
c_{\lambda\mu}(k)M_\mu\,\quad (c_{\lambda\mu}(k)\in \mathbb C);$
\\
\item[\rm{(ii)}] $\displaystyle L_k P_\lambda(k) = \langle \lambda, \lambda
+2\rho(k)\rangle P_\lambda(k)$,\quad
where
\begin{equation}\label{Laplace} L_k = \Delta_{\mathfrak a} + \sum_{\alpha\in R_+}
k_\alpha \,\text{coth}\frac{\langle \alpha, t\rangle}{2}\, \partial_\alpha\,.
\end{equation}
\end{enumerate}
\end{definition}
\noindent
Note that \eqref{Laplace} just gives the $W$-invariant part of the
Heckman-Opdam Laplacian, which is
%Tom end insert
given by restriction to $W$-invariant functions of
\[ \sum_{i=1}^n T_{\xi_i}(k)^2 -|\rho(k)|^2, \]
with an arbitrary  orthonormal basis $\{\xi_1, \ldots, \xi_n\}$  of $\mathfrak a$. The operator $L_k$ generalizes the radial part of the Laplace-Beltrami operator on a Riemannian symmetric space of the non-compact type.

Let us point out that in the definition of the Jacobi polynomials, condition (ii) is frequently replaced by an orthogonality condition. 
As remarked in Proposition 8.1 of \cite{H1}, both sets of conditions are equivalent. 
Note also that in \cite{H1}, the Jacobi polynomials are indexed by $-P_+$ instead of $P_+$, which leads to a different sign in~(ii). 

According to equation (4.4.10) of \cite{HS}, the $P_\lambda(k)$ 
can be expressed in terms of the hypergeometric function via
\begin{equation}\label{Tom1}
F_{\lambda +\rho}(k; t) = \,c(\lambda+\rho,k) P_\lambda(k;t),
\end{equation}
where $c(\lambda,k)$ is the generalized $c$-function 
as defined in \cite{HS}, Definition 3.4.2.
As the polynomial $P_0(k)$ is a constant, it follows that
\begin{equation}\label{constF}
F_\rho(k;t) = 1. 
\end{equation}

%%%%%%%%%%%%%%%%%%%%

\section{Some estimates for $G$ and $F$}

The growth behavior and asymptotic properties of the Cherednik kernel $G$ and the hypergeometric function $F$ have been studied in detail in \cite{O1} as well as  in \cite{S}, where the precise asymptotic behavior in the space variable was determined. 
We recall the
following  results: 

\begin{lemma}\label{F_estimates1}(\cite{O1}) Let $k\geq 0.$ Then for all $\lambda\in \mathfrak a_{\mathbb C}$ and all $t\in \mathfrak a, $
\begin{equation}\label{Opdamestim} |G_\lambda(k;t)| \,\leq \,\sqrt{|W|}\cdot e^{\text{max}_{w\in W}\text{Re}\langle w\lambda,t\rangle}.\end{equation}
\end{lemma}

\begin{lemma}\label{F_estimates2}(\cite{S}) Let $k\geq 0.$ Then \parskip=-1pt
 \begin{enumerate}\itemsep=-1pt
\item[\rm{(1)}] For $\lambda\in \mathfrak a,$ the kernel $G_\lambda(k;\,.\,) $ is real and strictly positive on $\mathfrak a$.
  \item[\rm{(2)}] $\displaystyle |G_\lambda(k;t)| \,\leq \,G_{\text{Re}\lambda}(k;t)\,$ for all $\lambda\in \mathfrak a_{\mathbb C}$ and $t\in \mathfrak a.$ 
 \end{enumerate}
\end{lemma}

By symmetrization over the Weyl group, one obtains the same properties and estimates  for the hypergeometric function $F$.

\smallskip
In \cite{S},  Opdam's estimate \eqref{Opdamestim} was substantially improved. In fact, it is shown there that for all $\lambda\in \mathfrak a $ and all $t\in \mathfrak a,$ 
\begin{equation}\label{Schapira_Estimate} G_\lambda(k;t) \leq G_0(k;t)\cdot e^{\text{max}_{w\in W}\langle w\lambda,t\rangle}\end{equation}
and that for fixed $k>0$, the kernel $G_0$ has the asymptotic behavior
\[ G_0(k;t) \,\asymp\,   \prod_{\alpha\in R_+^0\vert \langle \alpha,t\rangle \geq 0} \bigl(1+\langle\alpha,t\rangle\bigr) e^{-\langle\rho,t_+\rangle}\]
where $R_+^0$ denotes the set of indivisible positive roots and $\, t_+$ is  the unique element from the orbit $Wt$ which is contained 
in $ \overline{\mathfrak a_+}.$

\bigskip
\noindent
The following result generalizes Schapira's estimate \eqref{Schapira_Estimate}.

\begin{theorem}\label{hyperestimate} Let $k\geq 0. $ Then for all 
$\lambda\in \mathfrak a,$ all $\mu \in \overline{\mathfrak a_+}$ and all $\,t\in \mathfrak a,$
\[ G_{\lambda + \mu}(k;t)\leq G_\mu(k,t)\cdot e^{\max_{w\in W}\langle w\lambda,  t\rangle}.\]
The same estimate holds for the hypergeometric function $F$ instead of $G.$
\end{theorem}

For $\mu=\rho\in  \overline{\mathfrak a_+}$ we obtain, in view of identity \eqref{constF} and of Lemma \ref{F_estimates2}, the following  

\begin{corollary}\label{Cor_F} Let $k\geq 0. $ Then for all $\,\lambda \in \mathfrak a_{\mathbb C}$ and all $ t\in \mathfrak a,$
\begin{equation}\label{estimF_1}
\big|F_{\lambda+\rho}(k;t)\big|\,\leq \,e^{\max_{w\in W}
{\rm Re}\langle w\lambda,t\rangle} . \end{equation} 
\end{corollary}

\begin{remarks}\quad
\begin{enumerate}
\item While the proof of \eqref{Schapira_Estimate} is by real-analytic methods and uses the Cherednik operators, 
we shall present a different approach, based on methods from complex analysis.
\item Remark 3.1 of \cite{S} implies the following asymptotics
for  $t\in \overline{\mathfrak a_+}ß,$ when $k> 0$ and some real $\lambda \in \overline{\mathfrak a_+}$ are fixed: 
\[ F_{\lambda+\rho}(k;t) \asymp\,  e^{\langle \lambda, t\rangle}.\]
For our purposes, it will however be important  to have an estimate which is 
uniform in $k$. 
\end{enumerate}
\end{remarks}

For the proof of Theorem \ref{hyperestimate}, we shall use the Phragm\'{en}-Lindel\"of principle, see e.g. 
Theorem 5.61 of \cite{T}: 

\begin{lemma}\label{Phragmen} (Phragm\'{en}-Lindel\"of). Let $f$ be holomorphic in an open neighborhood of the right half plane $\,H=\{z\in \mathbb C: \text{Re}\,z \geq 0\},$ and suppose that $f$ satisfies
\[ |f(iy)|\leq M \quad \forall y\in \mathbb R\]
 and, as $|z|=r\to \infty,$ 
\[ f(z) = O(e^{r^{\beta}}) \]
for some $\beta  <1$, uniformly in $H$. Then actually $\,|f(z)|\leq M \,$ for all $z\in H.$
\end{lemma}

\bigskip
\begin{proof}[Proof of Theorem \ref{hyperestimate}]
Fix $t\in \mathfrak a$ and denote again by $t_+$ the unique element from the orbit $Wt$ which is contained 
in $ \overline{\mathfrak a_+}.$ Further, put
\[S:=\{ \lambda\in \mathfrak a_{\mathbb C}: \,\text{Re}\, \lambda \in \mathfrak a_+\}.\] 
The geometry of root systems  implies that for $\lambda\in \overline S$ and all $w\in W,$ 
\[ \langle w \,\text{Re} \lambda, t\rangle \leq 
\langle \text{Re}\lambda, t_+\rangle.\]
Fix now $w\in W$ and consider the function
\[ f(\lambda):= e^{-\langle \lambda,t_+\rangle}\cdot \frac{G_{w\lambda+\mu}(k;t)}{G_\mu(k;t)}\] 
which is holomorphic on $\mathfrak a_{\mathbb C}.$ We shall investigate $f$ on the closure $\overline S$ of $S.$ 
By part (2) of  Lemma \ref{F_estimates2}, we have
\[ |f(\lambda)|\leq f(\text{Re}\,\lambda).\]
Hence for $\lambda\in \overline S,$  Lemma \ref{F_estimates1}  leads  to the estimate
\begin{equation}\label{estim_unif} |f(\lambda)| \,\leq\, e^{-\langle\text{Re}\lambda, t_+\rangle}\cdot \frac{G_{w\text{Re}\lambda + \mu}(k;t)}{G_\mu(k;t)}
\,\leq\, \sqrt{|W|}\cdot  \frac{e^{\langle\mu,t_+\rangle}}{G_\mu(k;t)}.\end{equation}
Note that the right side is independent of $\lambda.$
Again by  Lemma \ref{F_estimates2},  we further obtain for real $\lambda\in \mathfrak a$ the uniform estimate
\begin{equation}\label{estim_imag} |f(i\lambda)| = \,\frac{|G_{iw\lambda+\mu}(k;t)|}{G_\mu(k;t)}\,\leq\, 1.\end{equation}
We claim that $|f|\leq 1\,$ on $\overline S.$ For this, fix a basis $\{\lambda_1, \ldots, \lambda_n\}\subseteq P_+$ of fundamental weights. Then each $\lambda\in \overline S$ has a unique expansion $\lambda = \sum_{i=1}^n z_i\lambda_i$ with 
$z_i\in H=\{z\in \mathbb C: \text{Re}\,z \geq 0\}.$ Consider first
$\lambda=z_1\lambda_1$ with $z_1\in H.$ In view of  estimates \eqref{estim_imag}
and \eqref{estim_unif},
we may apply Lemma~\ref{Phragmen} with $\beta = 0,$ thus obtaining
\[ |f(z_1\lambda_1)|\leq 1 \quad \forall z_1\in H.\]
We proceed by induction:
Suppose, for $1\leq m < n,$ that 
\[ |f(z_1\lambda_1+\ldots + z_m\lambda_m)|\leq 1 \quad \forall \,z_1, \ldots, z_m\in H.\]
Consider $h(z_{m+1}):= f(z_1\lambda_1+\ldots + z_m\lambda_m +z_{m+1}\lambda_{m+1})$ for $z_{m+1}\in H.$ This function is uniformly bounded on $H$ according to \eqref{estim_unif}, and for purely imaginary $z_{m+1}\in i\mathbb R$ we have
\[|h(z_{m+1})|\leq \big|f\bigl(\text{Re}(z_1\lambda_1+\ldots + z_m\lambda_m + z_{m+1}\lambda_{m+1})\bigl)\big| = 
  |f\bigl(\text{Re} z_1\cdot\lambda_1+\ldots + \text{Re}z_m \cdot\lambda_m\bigr)| 
\]
which is less or equal to $1$ by our induction hypothesis. 
By Lemma \ref{Phragmen}, we conclude that $|h(z)|\leq 1\,$ for all $z\in H$.
Thus, induction shows that $|f(\lambda)|\leq 1\,$ for all $\lambda\in \overline S,$ and in particular for all $\lambda\in \overline{\mathfrak a_+}.$ If $\lambda \in \mathfrak a$ is arbitrary, just use the fact that
$\lambda = w\lambda^\prime $ with some $w\in W$ and $\lambda^\prime\in \overline{\mathfrak a_+}.$ This implies the assertion.

\end{proof}

%%%%%%%%%%%%%%%%%%%%

\section{Limit transition for Jacobi polynomials of type BC }

Let $\mathfrak a=\mathbb R^n$ with the usual Euclidian scalar product and denote by $(e_i)_{ i=1, \ldots, n}$  the standard basis of $\mathbb R^n$. 
We consider the root system $BC_n$ in $\mathbb R^n$ with the positive subsystem 
\[  BC_n^+ = \{ e_i,2e_i, \, 1\leq i \leq n\} \cup \{e_i \pm e_j, \,1\leq i<j\leq n\},
\]
 as well as the root system $A_{n-1}$ in the linear subspace 
\[\mathbb R^n_0:= \{ t\in \mathbb R^n: t_1+\ldots + t_n  = 0\}\]
with the positive subsystem 
\[ A_{n-1}^+ = \{e_i - e_j, \,1\leq i<j\leq n\}.\]

The Jacobi polynomials associated with these root systems
%Tom begin insert
(following Definition \ref{JacobiPolDef})
%Tom end insert
as well as their relationship have been widely studied; see in particular  \cite{BO},
\cite{BF}, \cite{H1} and \cite{H2}.
We recall the fundamental facts:
Let \begin{equation}\label{projection}
 \pi(t) := t-\frac{1}{n}\langle t,\omega_n\rangle \omega_n \end{equation}
with 
\[\omega_n = e_1 + \ldots +e_n\]
denote the orthogonal projection of $\mathbb R^n$ onto $\mathbb R^n_0$.
The cone of dominant weights of $BC_n$ is 
\[ P_+^{BC} = \{(\lambda_1, \ldots \lambda_n) \in \mathbb Z_+^n: \lambda_1 \geq \ldots \geq \lambda_n\},\]
and the dominant weights of $A_{n-1}$ are given by
\[ P_+^A = \pi\bigl(P_+^{BC}\bigr).\]
For abbreviation, we write $P_+ := P_+^{BC}, $ which is just the set of
partitions of length~$n$. The dominance order
%on $P_+$ is given by
%\[ \lambda\leq \mu \,\Longleftrightarrow\,
%\sum_{j=1}^i \lambda_j \leq \sum_{j=1}^i \mu_j\,, \quad i=1, \ldots, n.\]
%
%Tom begin insert
and inclusion order on $P_+$ are respectively given by
\begin{align*}
\lambda\leq \mu \,&\Longleftrightarrow\,
\sum_{j=1}^i \lambda_j \leq \sum_{j=1}^i \mu_j\,, \quad i=1, \ldots, n,\\
\lambda\subseteq\mu\,&\Longleftrightarrow\,
\lambda_i\le\mu_i\,,\quad i=1,\ldots,n.
\end{align*}
%Tom end insert
For the  $A_{n-1}$-case, we take  a real parameter $\kappa\geq 0$ and consider
 the \emph{monic} Jack polynomials $j_\lambda^\kappa$ in $n$ variables which are indexed by partitions $\lambda\in P_+$  and are uniquely characterized by the following conditions:  \vspace{-1pt}
\begin{enumerate}\itemsep=-1pt
\item[\rm{(1)}] $j_\lambda^\kappa$ is homogeneous of degree $|\lambda|$ and of the form
\[j_\lambda^\kappa = m_\lambda + \sum_{\mu<\lambda} c_{\lambda\mu}(\kappa) m_\mu\]
where $\mu < \lambda$ refers to the dominance order on $P_+$ and the $m_\lambda, \,\lambda\in P_+$, are the monomial symmetric polynomials
\[m_\lambda(x) = \sum_{\mu\in S_n\lambda} x^\mu \quad (x\in \mathbb R^n).\]
\item[\rm{(2)}] $j_\lambda^\kappa$ is an eigenfunction of the operator
\[ D_\kappa = \sum_{i=1}^n x_i^2\,\frac{\partial^2}{\partial x_i^2}\, + 2\kappa \sum_{i\not=j} \frac{x_i^2}{x_i-x_j}\frac{\partial}{\partial x_i}.\] 
\end{enumerate}
In fact, the Jack polynomials satisfy
\[ D_\kappa j_\lambda^\kappa = d_\lambda(\kappa) j_\lambda^\kappa \quad \text{with }\, d_\lambda(\kappa) =
\sum_{i=1}^n \lambda_i\bigl(\lambda_i-1+2\kappa(n-i)\bigr),\]
see \cite{Ha} or \cite{St}.
For $\kappa=0,$ we have $j_\lambda^0 = m_\lambda,$ while for
 $\kappa >0,$ the polynomial $j_\lambda^\kappa(x)$ coincides up to  constant positive factor with the Jack polynomial $J_\lambda(x;1/\kappa)$ in standard normalization as introduced in \cite{St}. 

\smallskip
The Heckman-Opdam Jacobi polynomials of type $A_{n-1}$ with multiplicity parameter $\kappa \geq 0$ 
are essentially Jack polynomials; according to  Proposition 3.3. of \cite{BO}, %they are given by
%Tom begin insert
the two types of polynomials are related by
%Tom end insert
\[ P_{\pi(\lambda)}^{A}(\kappa; t) = j_\lambda^\kappa(e^{t}) \quad \text{where }\,e^t = (e^{t_1}, \ldots, e^{t_n}),\,\, t=(t_1, \ldots, t_n) \in \mathbb R^n_0\,.\]
Notice that the homogeneity of the Jack polynomials implies that for arbitrary $t\in \mathbb R^n,$ 
\begin{equation}\label{jackhom} j_\lambda^\kappa(e^{t}\,) = e^{|\lambda|(t_1+\ldots +t_n)/n}\cdot j_\lambda^\kappa(e^{\pi(t)}).\end{equation}

The Heckman-Opdam Jacobi polynomials of type $BC_n$ are
 parameterized by a multiplicity function 
$k=(k_1, k_2, k_3)\geq 0$ on $BC_n$, where $k_1$ stands
 for the parameter on $e_i$, $k_2$ for the parameter on $2e_i$ and $k_3$ for the parameter on $e_i\pm e_j$. 
Let $L_k^{BC}$ be the associated operator \eqref{Laplace}
of type $BC_n$. The corresponing eigenvalue (see \cite{Ha}) is
\[ e_\lambda(k):= d_\lambda(k_3) + (k_1+2k_2+1)|\lambda|,\]
with $d_\lambda$ as above. We then obtain from  \cite{Ha}
the following representation of the $BC_n$-type Jacobi
 polynomials $P_\lambda^{BC}(k)$ in terms of the Jack polynomials $j_\lambda^{k_3}:$

\begin{proposition}
For all $\lambda,k,t$ as above,
\begin{equation}\label{BCA} P_\lambda^{BC}(k;t) = 4^{|\lambda|}
\prod_{\mu\subset\lambda} \frac{L_k^{BC} - e_\mu(k)}{e_\lambda(k) -
  e_\mu(k)}\,j_\lambda^{k_3}\bigl(-{\rm sinh}^2\bigl(\frac{t}{2}\bigr)\bigr).\end{equation}
Here $\displaystyle\,
{\rm sinh}^2\bigl(\frac{t}{2}\bigr)$ is  understood component-wise, 
and  $\mu\subset\lambda$ means that $\mu\not=\lambda$ and $\mu_i\leq
\lambda_i$ for all $i.$
\end{proposition}

\begin{proof}
Denote the right hand side of (\ref{BCA}) by  $\tilde P_\lambda^{BC}(k;t)$.
It follows from  relation (13) of \cite{Ha} that $\tilde P_\lambda^{BC}(k;t)$
is equal to $P_\lambda^{BC}(k;t)$ up to a multiplicative constant.
In order to identify this constant, 
 we compare the leading  terms of both polynomials in  the expansion with respect to the 
 monomial symmetric functions $ M_\mu^{BC}$ of type $BC$ as defined in \eqref{monomial}.
In fact, 
\[ 4^{|\lambda|}m_\lambda\bigl({\rm sinh}^2\frac{t}{2}\bigr) = M_\lambda^{BC}(t) \,+\sum_{\mu<\lambda} b_{\lambda\mu}M_\mu^{BC}(t)\]
with certain constants $b_{\lambda\mu}$ (see \cite{SK}, the last displayed formula on p. 383 with $t=i\theta$).
%Taking the characterization of the Jack polynomials  $j_\lambda^\kappa$
%as well as part 1 of the characterization of the  $\tilde P_\lambda^{BC}(k;t)$
%and the definition of  $t_i$ on p. 1580 of  \cite{Ha} into account,
%
%Tom begin insert
Next we use
the characterization of the Jack polynomials $j_\lambda^\kappa$ given above
and also, from p.~1580 of \cite{Ha}, the first part of the characterization of the 
$\tilde P_\lambda^{BC}(k;t)$ and the definition of  $t_i$. Then
%Tom end insert
we conclude that
\begin{equation}\label{tilde-p-lambda} \tilde P_\lambda^{BC}(k;t)
=M_\lambda^{BC}(t) \,+\sum_{\mu<\lambda}  d_{\lambda\mu}M_\mu^{BC}(t) 
\end{equation}
with certain coefficients $ d_{\lambda\mu}$. Therefore $\tilde P_\lambda^{BC}(k;t) = P_\lambda^{BC}(k;t)$ as claimed.
\end{proof}

We notice that representations such as (\ref{BCA}) 
were already observed by Macdonald \cite{M} and were used in
 \cite{SK} for limit transitions between different families of orthogonal polynomials. 

\medskip
{}From \eqref{BCA}, we shall deduce the following limit result.

\begin{theorem}\label{polynomlimit} Fix a parameter $0\leq a\leq \infty$ and consider $k=(k_1,k_2,k_3)$ where $k_3\geq 0$ is fixed. 
Then 
\begin{equation}\label{general}
\lim_{\substack{ k_1+k_2 \to \,\infty\\k_1/k_2\,\to \,a}}
P_\lambda^{BC}(k;t) \,=\, 4^{|\lambda|}\cdot j_\lambda^{k_3}(x(t)),
\end{equation}
where the transform $t\mapsto x(t), \mathbb R^n\to \mathbb R_+^n$ is given by 
\[  x_i(t)= \gamma_a+ {\rm sinh}^2\bigl(\frac{t_i}{2}\bigr), \quad \gamma_a = \frac{a+1}{a+2}\]
with the understanding that $\gamma_\infty =1.$
The convergence in  \eqref{general} is locally uniform in $t\in \mathbb R^n.$ 
Especially if $a=\infty, $ then
\begin{align}\label{specialcase}
\lim_{\substack{ k_1+k_2 \to \,\infty\notag\\k_1/k_2\,\to \,\infty}}
P_\lambda^{BC}(k;t) \,=& \,4^{|\lambda|}j_\lambda^{k_3}\bigl({\rm cosh}^2\frac{t}{2}\bigr) \\
=\,& \Bigl(\prod_{i=1}^n 4\,{\rm cosh}^2\frac{t_i}{2}\Bigr)^{|\lambda|/n}\cdot P_{\pi(\lambda)}^A\bigl(k_3;\pi\bigr(\log{\rm cosh}^2\frac{t}{2}\bigr)\bigr)  \end{align}
\end{theorem} 

\noindent
The case $a=\infty$ occurs for instance if $k_2, \, k_3\geq 0$ are fixed and $k_1\to \infty.$

\begin{proof} We split the coordinate transform $t\to x(t)$ and consider first the transform
\[ y_i =  -{\rm sinh}^2\frac{t_i}{2}\]
which is frequently used in the $BC$-setting. In $y$-coordinates, the operator 
$L_k^{BC}$ becomes
%\[\widetilde L_k^{BC} \,=\,
%\sum_{i=1}^n y_i(y_i-1)\frac{\partial^2}{\partial y_i^2} \,-\, \sum_{i=1}^n
%\bigl(k_1+k_2+\frac{1}{2} -(k_1+2k_2+1)y_i\bigr)\frac{\partial}{\partial y_i}
%\,+\,2k_3\sum_{i\not=j}\frac{y_i(y_i-1)}{y_i-y_j}\,
%\frac{\partial}{\partial y_i},\]
%
%Tom begin insert
\begin{multline*}
\widetilde L_k^{BC} \,=\,
\sum_{i=1}^n y_i(y_i-1)\frac{\partial^2}{\partial y_i^2} \,-\, \sum_{i=1}^n
\bigl(k_1+k_2+\frac{1}{2} -(k_1+2k_2+1)y_i\bigr)\frac{\partial}{\partial y_i}\\
+\,2k_3\sum_{i\not=j}\frac{y_i(y_i-1)}{y_i-y_j}\,
\frac{\partial}{\partial y_i},
\end{multline*}
%Tom end insert
see Section 4 of \cite{BO} (or also \cite{Ha}, Section 2.3).
Next, we carry out the linear transform
$\, \displaystyle x_i = \gamma_a - y_i,$
under which $\widetilde L_k^{BC}$ becomes
\begin{align*}\label{LBC_x} \widehat L_k^{BC}\, =  &\,
\sum_{i=1}^n (\gamma_a-x_i)(\gamma_a-1-x_i)\frac{\partial^2}{\partial x_i^2}  \,+\, 
2k_3\sum_{i\not=j}\frac{(\gamma_a-x_i)(\gamma_a-1-x_i)}{x_i-x_j}\,\frac{\partial}{\partial x_i} \\
& +\,
\sum_{i=1}^n \bigl(k_1+k_2+\frac{1}{2} -(k_1+2k_2+1)(\gamma_a-x_i)\bigr)\frac{\partial}{\partial x_i}.
\end{align*}
Equation \eqref{BCA} thus writes
\begin{equation}\label{BCA2}  P_\lambda^{BC}(k;t) = 4^{|\lambda|}\Bigl(\prod_{\mu\subset\lambda} \frac{\widehat L_k^{BC} - e_\mu(k)}{e_\lambda(k) - e_\mu(k)}\,j_\lambda^{k_3}\Bigr)(x)
\end{equation}
with $x=x(t).$ As $k_1+k_2\to\infty$, we have
\[ e_\lambda(k) \,\sim \,|\lambda|(k_1+2k_2).\]
If in addition $k_1/k_2\to a,$ then 
\[ \frac{k_1+k_2}{k_1+2k_2} \to \gamma_a\,.\]
Now let $\mu\subset\lambda$. Then $|\mu|<|\lambda| $ and for  $f\in C^\infty(\mathbb R^n)$ we obtain,
as $(k_1, k_2)\to \infty$ in the required way, 
\[ \frac{\widehat L_k^{BC}-e_\mu(k)}{e_\lambda(k)-e_\mu(k)}f(x) \,\longrightarrow \, 
\frac{1}{|\lambda|-|\mu|}\Bigl(\sum_{i=1}^n x_i\frac{\partial}{\partial x_i}\,-|\mu|\Bigr)f(x).\]
For $f$ a symmetric polynomial, the convergence is locally uniform in $x\in \mathbb R^n$. 
In our case, $f= j_\lambda^{k_3}$ is homogeneous of degree $|\lambda|.$ Thus
\[ \sum_{i=1}^n x_i\frac{\partial}{\partial x_i} j_\lambda^{k_3}(x) = |\lambda|\cdot j_\lambda^{k_3}(x)\]
and therefore
\[ \frac{\widehat L_k^{BC}-e_\mu(k)}{e_\lambda(k)-e_\mu(k)}\,j_\lambda^{k_3} \,\longrightarrow \, j_\lambda^{k_3}\]
 locally uniformly, for each $\mu\subset\lambda$. Iteration according to formula \eqref{BCA2} yields
\[P_\lambda^{BC}(k;t) \,\longrightarrow\, 4^{|\lambda|}\cdot j_\lambda^{k_3}(x(t)),\]
locally uniformly in $t$ which completes the proof of relation
\eqref{general}.
 Finally, in the setting of relation \eqref{specialcase} we have  $\gamma_a=1,$ and the claimed limit result follows from formula \eqref{jackhom}.
%Tom: I removed this blank line in the source
\end{proof}

%Tom begin insert
\begin{remark}
Theorem \ref{polynomlimit}
was already stated without proof as Theorem 1 in \cite{K2}.
There it was based on an unpublished manuscript of R.~J. Beerends and
Koornwinder. The proof given in this manuscript uses the coefficients
$c_{\lambda,\mu}$ in
\[
P_\lambda(k)=\sum_{\mu\le\lambda}c_{\lambda,\mu}(k) e^\mu
\]
with $c_{\lambda,\lambda}=1$ and $c_{\lambda,w\mu}=c_{\lambda,\mu}$ ($w\in W)$
(which is equivalent to Definition \ref{JacobiPolDef}(i)).
By \eqref{Laplace} and (ii) of this Definition there follows a recurrence
relation for the $c_{\lambda,\mu}$ which determines them uniquely with the
given initial value $c_{\lambda,\lambda}=1$. Then it is shown that the
coefficients in the recurrence relation for the $c_{\lambda,\mu}$
in case $BC$ tend in the limit under consideration
to the corresponding coefficients in the
recurrence relation for the Jack case. This essentially involves
the asymptotics of the operator $L_k^{BC}$ and the eigenvalue $e_\lambda(k)$,
just as we used in the proof of Theorem \ref{polynomlimit}.
\end{remark}

\begin{remark}
It follows from Macdonald \cite{M0} (see also \cite[(5.3)]{BO},
\cite[Th\'eor\`eme~3]{La} and
\cite[p.1580]{Ha})
that from
Definition \ref{JacobiPolDef} an equivalent definition is obtained by replacing
condition (i) by
\medbreak${\rm(i)}'$\quad
$\displaystyle P_\lambda^{BC}(k;t)=\sum_{\mu\subseteq\lambda}
u_{\lambda\mu}\,j_\lambda^{k_3}\big(-\sinh^2(\tfrac12 t)\big)$,\quad
$u_{\lambda\lambda}=(-4)^{|\lambda|}$.
\medbreak\noindent
Macdonald \cite{M0} also obtained a recurrence relation for the coefficients
$u_{\lambda\mu}$. This can be used in order to give a third proof of
Theorem \ref{polynomlimit}. Furthermore, in combination with the homogeneity of the
Jack polynomials, ${\rm(i)}'$ yields another limit from $BC_n$ type
Jacobi polynomials to Jack polynomials:
\begin{equation}
\lim_{r\to\infty}e^{-|\lambda|r\langle t,\omega\rangle}
P_\lambda^{BC}(k;t+r\omega)=j_\lambda^{k_3}(e^t).
\end{equation}
This can be further specialized as a limit to $A_{n-1}$ type Jacobi polyomials.
Then it is the $q=1$ analogue of a limit from Macdonald-Koornwinder
polynomials to $A_{n-1}$ type Macdonald polynomials
given by van Diejen \cite[Section 5.2]{vD}.
\end{remark}
%Tom end insert

%%%%%%%%%%%%%%%%%%%%

\section{Limit transition for hypergeometric functions of type BC }

We now extend the above limit transition to the associated hypergeometric
functions, where we restrict our attention  to the case $a=\infty.$ 

For abbreviation, we write $C_B$ for
the closed Weyl chamber associated with the positive system $BC_n^+,$ i.e. 
\[ C_B = \{t\in \mathbb R^n: t_1 \geq  \ldots \geq  t_n \geq  0\}.\]
Observe that under the projection $\pi: \mathbb R^n\to \mathbb R^n_0$, the chamber $C_B$ is mapped 
onto the closed Weyl chamber associated with the positive subsystem $A_{n-1}^+$ of $A_{n-1}$.

Again, we consider $k=(k_1, k_2, k_3) $ where $k_3\geq 0$ is fixed. We also
recapitulate that the half-sums (\ref{general-halfsum}) of positive roots for
$BC_n$ and $A_{n-1}$ are given by
\begin{equation} \rho_{BC}(k)= \sum_{i=1}^n (k_1+2k_2+2k_3(n-i))e_i  \quad\text{and}\quad 
\rho_A(k_3)=k_3 \sum_{i=1}^n (n+1-2i)e_i\,. 
\end{equation}

\begin{theorem}\label{asymptot_cont}
 For
   each $t\in \mathbb R^n$ and $\lambda\in \mathbb C^n,$ 
\begin{multline*}\lim_{\substack{ k_1+k_2 \to \,\infty\\
                                k_1/k_2\,\to \,\infty}}
F_{BC}(\lambda+\rho_{BC}(k),k;t) \\
= \, \prod_{i=1}^n \bigl({\rm cosh}^2
\frac{t_i}{2}\bigr)^{\langle\lambda,\omega_n\rangle/n}\cdot
F_A\bigl(\pi(\lambda) + \rho_A(k_3), k_3; \pi\bigl(\log {\rm
cosh}^2\frac{t}{2}\bigr)\bigr).\end{multline*}
The convergence is locally uniform with respect to $\lambda$. 
\end{theorem}

Notice that in this situation, $\rho_{BC}(k)\to \infty$. 
The proof of Theorem \ref{asymptot_cont} will be based on Theorem \ref{polynomlimit} above and the following
well-known theorem of Carlson (see e.g. [T], Theorem 5.81):

\smallskip\noindent
\begin{theorem}(Carlson's Theorem)
Let $f$ be a function which is holomorphic in a neighborhood
 of $\{ z\in \mathbb C: \text{Re}\, z \geq 0\}$ 
and satisfies $\, f(z) = O(e^{c|z|})$ for some constant $c<\pi$.
 Suppose that $f(n) = 0$ for all $n\in \mathbb N_0$. Then $f$ is identically zero.
\end{theorem}

\begin{proof}[Proof of Theorem \ref{asymptot_cont}]  Let $\mathcal K_+:= \{k=(k_1, k_2, k_3)\in \mathbb R^3: 
k_i\geq 0 \,\,\forall \,i\}$ and fix some $t\in \mathbb R^n$. By the $BC$-symmetry of both sides, 
we may assume that $t\in C_B\,.$ 
For $k\in \mathcal K_+$ define 
\[f_k(\lambda) := \, e^{-\langle \lambda, t\rangle}\cdot F_{BC}(\lambda+\rho_{BC}(k),k;t)\]
and 
\[ g(\lambda):= \, e^{-\langle \lambda, t\rangle}\cdot\prod_{i=1}^n 
\bigl({\rm cosh}^2 \frac{t_i}{2}\bigr)^{\langle\lambda,\omega_n\rangle/n}\cdot  F_A\bigl(\pi(\lambda) + \rho_A(k_3), k_3; 
\pi\bigl(\log {\rm cosh}^2\frac{t}{2}\bigr)\bigr).\]
The functions $f_k$ and $g$ are holomorphic on $\mathbb C^q$. 
Corollary \ref{Cor_F} readily implies that the family $\{f_k: \,k\in \mathcal K_+\}$ is locally bounded on 
$\mathbb C^q$ and  uniformly bounded on the set $\,S:= \{ \lambda\in \mathbb C^q:  \,\text{Re}\,\lambda \in C_B\}; $ 
indeed,
as $t\in C_B$ we obtain
\begin{equation}\label{uniform_S} |f_k(\lambda)|\,\leq \,1 \quad \forall \, \lambda\in S.
 \end{equation}
Now let  $(k(j))_{j\in \mathbb N}\subset \mathcal K_+\,$ be a sequence of multiplicities such that $k(j)_3 = k_3$ 
with fixed $k_3\geq 0 $ and 
$\, k(j)_1+k(j)_2 \to +\infty, \, k(j)_1/k(j)_2 \, \to +\infty.\,$ 
For abbreviation, we write 
\[ f_j := f_{k(j)}, \,\, j\in \mathbb N.\]
We have to show that $\, f_j\to g\,$ locally uniformly on $\mathbb C^q$. 
By Montel's theorem in several complex variables (see for instance \cite{G}), each locally bounded sequence of holomorphic functions on $\mathbb C^q$ has a subsequence which converges locally uniformly to some limit function which is again holomorphic on $\mathbb C^q$. 
It therefore suffices to verify the following condition: 

(M) \enskip If $(f_{j_\nu})$ is a subsequence of $(f_j)$ 
such that $f_{j_\nu} \,\to h \,$ locally uniformly on $\mathbb C^q$ for some
$h$, then 
$h=g$ on $\mathbb C^q.$ 

\smallskip
Suppose that $(f_{j_\nu})$ is a subsequence with  $f_{j_\nu} \,\to h \,$ locally uniformly on $\mathbb C^q$.
According to Theorem \ref{polynomlimit} together with \eqref{Tom1} and
$F_{\lambda +\rho}(k;0)=1$, we have
\[ f_{j_\nu}(\lambda) \,\to \, g(\lambda)\]
for all dominant weights $\lambda\in P_+\,.$ Therefore $\,h(\lambda) = g(\lambda) \,$ for all $\lambda\in P_+\,.$ 
Consider again the set $S$. We claim that 
\begin{equation}\label{hgrechts} h(\lambda) = g(\lambda) \quad \forall \lambda \in S.\end{equation}
Once  this is shown, the identity theorem will imply that $h=g$ on $\mathbb C^q$, and the verification of condition 
$(M)$ will be accomplished. For the proof of \eqref{hgrechts} we shall apply Carlson's theorem to $g-h$ on 
$S,$ which requires  suitable growth bounds on the involved functions. First, $h$ is the locally uniform limit of 
the sequence $f_{j_\nu}$ which is uniformly bounded on $S$ according to \eqref{uniform_S}. Hence
\[ |h(\lambda)|\,\leq 1 \quad \forall \lambda\in  S.\]
For an estimate of $g$ on $S,$ note that $\,\text{Re}\,\pi(\lambda)$ is contained in the closed positive chamber 
associated with $A_{n-1}^+$ for each $\lambda\in  S.$ Application of  Corollary \ref{Cor_F}  therefore yields
\[ \Big\vert e^{-\langle \pi(\lambda), \pi(\text{log}({\rm cosh}^2\frac{t}{2}))\rangle} \cdot
 F_A\bigl(\pi(\lambda)+ \rho_A(k_3), k_3; \pi\bigl(\log {\rm cosh}^2\frac{t}{2}\bigr)\bigr)\Big\vert\,\leq \, 1\]
for all $\lambda\in S.$ Let us call the function on the left $E(\lambda)$ and write
\[ |g(\lambda)| \,= \, \big\vert e^{-\langle\lambda, t\rangle}\cdot e^{\langle \pi(\lambda), 
\pi(\text{log}({\rm cosh}^2\frac{t}{2}))\rangle}\cdot\prod_{i=1}^n \bigl({\rm cosh}^2 
\frac{t_i}{2}\bigr)^{\langle\lambda,\omega_n\rangle/n}\big\vert \cdot E(\lambda).\]
As 
\[ \langle \pi(x), \pi(y)\rangle\, = 
\,\langle x,y\rangle -\,\frac{1}{n} \langle x,\omega_n\rangle \langle y,\omega_n\rangle  
\quad \forall x,y\in \mathbb R^n,\]
we obtain 
\[
 |g(\lambda)| \,= \,\big\vert e^{-\langle\lambda, t\rangle}\cdot e^{\langle \lambda, 
\text{log}({\rm cosh}^2\frac{t}{2})\rangle}\big\vert\cdot E(\lambda)\, 
 \leq \,  \prod_{i=1}^n \bigl( e^{-t_i }\,{\rm cosh}^2\frac{t_i}{2}\bigr)^{{\rm Re}\, \lambda_i}
 \]
and therefore 
\[ |g(\lambda)|\leq 1  \quad \forall \lambda\in  S.\]
Summing up, we have 
\[|g-h|\leq 2\,\,\text{ on }\, S \,\, \text{ and }\, (g-h)(\lambda) = 0\,\,\,\forall \, \lambda \in P_+.\]
As in the proof of Theorem \ref{hyperestimate}, we fix a set of fundamental weights 
$ \{ \lambda_1, \ldots, \lambda_n\}\subset P_+\,$ and  write $\lambda \in S$ as 
$\,\lambda= \sum_{i=1}^n z_i\lambda_i\,$ with coefficients $z_i\in \{ z\in \mathbb C: \text{Re}\, z\geq 0\}.$ 
Then successive use of Carlson's Theorem with respect to the variables 
$z_1, \ldots, z_n$ shows that actually $g-h = 0$ on $S$. 

\end{proof}

\section{Limit transition for spherical functions of noncompact  Grassmann manifolds}
\subsection{Spherical functions of non-compact Grassmannians}
For specific multiplicities, hypergeometric functions of type $BC$ occur as
spherical
 functions of non-compact Grassmann manifolds. This was the starting point for
 the 
construction of hypergroup convolution algebras with hypergeometric functions
as 
characters in \cite{R}. Let us recall this connection. For each of the fields 
$\mathbb F= \mathbb R, \mathbb C, \mathbb H$ we consider 
the Grassmann manifolds $\mathcal G_{p,q}(\mathbb F)= G/K$ where $G$ is one of the 
groups $SO_0(p,q), \, SU(p,q) $ or $Sp(p,q)$ with maximal compact subgroup
 $K=SO(p)\times SO(q), \, S(U(p)\times U(q))$ or $Sp(p)\times Sp(q),$ where
 we assume that $p> q.$ We regard $G$ and $K$ as  subgroups of the indefinite unitary group 
$U(p,q;\mathbb F)$ over $\mathbb F$. 
The Lie algebra $\mathfrak g$ of $G$ has the Cartan decomposition $\mathfrak g = \mathfrak k \oplus \mathfrak p$ where $\mathfrak k$ is the Lie algebra of $K$ and $\mathfrak p$ consists of the $(p+q)$-block matrices
\[ \begin{pmatrix} \,0 & X \\
\overline X^t & 0 
   \end{pmatrix}, \quad X\in M_{p,q}(\mathbb F).\]
  As a maximal abelian subspace of $\mathfrak p$ we choose
   \[ \mathfrak a = \left\{ H_t = \begin{pmatrix} 0_{p\times p}& \begin{matrix}\underline t \\
                            0_{(p-q)\times q}\end{matrix}\\
\begin{matrix} \,\underline t& 0_{q\times(p-q)}\end{matrix}& 0_{q\times q}\end{pmatrix}, \,\, t\in \mathbb R^q \right\}\]
where $\, \underline t := \text{diag}(t_1, \ldots, t_q)$ is the $q\times q$ diagonal matrix corresponding to $t.$ 

The  restricted root system $\Delta=\Delta(\mathfrak g, \mathfrak a)$ is of type $BC_q$ with the understanding that zero is allowed as a multiplicity on the long roots. 
 We identify $\mathfrak a^*$ with $\mathfrak a$ via the Killing form and $\mathfrak a$  with $\mathbb R^q$ via 
the mapping $H_t\mapsto t$. Under this identification, the Killing form corresponds to a constant multiple of the Euclidean
scalar product on $\mathbb R^q$, and
 \[ \Delta = BC_q = \{ \pm e_i, \pm 2e_i, \pm e_i \pm e_j, \,\, 1\leq i < j \leq q\} \subset \mathbb R^q.\]
 The geometric multiplicities of the roots are given by 
 \[m_\alpha = \begin{cases} d(p-q) & \text{for }\,\alpha = \pm e_i\\
                                              d-1 & \text{for }\,\alpha = \pm 2e_i\\
                                              d & \text{for }\,\alpha = \pm e_i\pm e_j.
              \end{cases}\]
              where $d= \dim_{\mathbb R}\mathbb F.$ 
We consider the spherical functions of  $G/K$  as functions on $A = \exp \mathfrak a$. 
Let $F_{BC}$ denote the hypergeometric function associated with 
$ R=BC_q$ and multiplicity  $k_{\alpha} = \frac{1}{2}m_\alpha$ ($m_\alpha$ as above), and denote by $\widetilde F_{BC}$ the hypergeometric function associated with the rescaled
 root system $ \widetilde R=2BC_q$ and multiplicity $\widetilde k_{2\alpha} = k_\alpha.$ 
Then according to Remark 2.3 of
\cite{H3}
and Lemma \ref{scaling}, 
the spherical functions of the Grassmannian $\mathcal G_{p,q}(\mathbb F)$ 
are given by
\begin{equation}\label{grassmann-spherical-bc}
 \varphi_\lambda(a_t) \,=\,  \widetilde F_{BC}(\lambda, \widetilde k; t)
 \,=\, F_{BC}(\lambda/2,k; 2t), \quad \lambda \in \mathbb C^q, \end{equation}
 where $$ t\in \mathbb R^q \quad\text{and}\quad 
a_t=e^{H_t}=\begin{pmatrix}{\rm cosh}\,\underline t&0&{\rm sinh}\, \underline t \\ 0&I_n&0  \\
 {\rm sinh}\,\underline t&0&{\rm cosh}\,\underline t
                            \end{pmatrix}.$$
The limit $k_1\to\infty$ in Theorem \ref{asymptot_cont} here 
corresponds to $p\to \infty$.
In order to  identify  the limit in this case, we  recapitulate some facts on spherical functions of type A.

\subsection{Spherical functions of type A}\label{subsection-a-fall}
 Consider the symmetric spaces $G/K$ where $G$ is one of the connected reductive groups 
$GL_+(q,\mathbb R),$ $GL(q,\mathbb C),$ $GL(q, \mathbb H)$
 with maximal compact subgroup $K=SO(q), \,U(q)$ and $Sp(q)$, 
respectively.  We have the Cartan decomposition  $G=KAK$ with 
\begin{equation}\label{A-a-fall}
 A= \exp \mathfrak a, \quad \mathfrak a=\{ \,\underline t = \text{diag}(t_1, \ldots, t_q),  \,\,\, t=(t_1, \ldots, t_q)\in \mathbb R^q \}.
\end{equation}
For the moment, we consider the spherical functions of $G/K$ as functions on $\mathfrak a$, where we 
identify $\mathfrak a \cong \mathbb R^q $
via $\,\underline t\,\mapsto t.$ The spherical functions of $G/K$ are then
 characterized as the continuous functions on $\mathbb R^q$ 
which are symmetric and satisfy the product formula
\begin{align} \label{prod_A}
\psi(t)\psi(s) = \int_K \psi\bigl(\log(\sigma_{sing}(e^{\underline t}\, k\,e^{\underline s}))\bigr)dk;\end{align}
here  $\sigma_{sing}(M) = (\sigma_1, \ldots,\sigma_q) \in \mathbb R^q$ denotes
the singular values of $M\in M_q(\mathbb F)$ ordered by size: $\sigma_1\geq \ldots \geq \sigma_q.$ 
The spherical functions of $G/K=GL(q, \mathbb F)/U(q,\mathbb F)$ are closely related to those of $G_1/K_1$  where $G_1$ is the corresponding semisimple 
group $SL(q, \mathbb F)$ and $K_1 = SU(q, \mathbb F).$ 
Indeed, consider the orthogonal projection $\pi: \mathbb R^q \to \mathbb R_0^q$ as in \eqref{projection}. 
In the same way as above, the spherical functions of $G_1/K_1$ 
may be characterized as the symmetric functions $\psi$ on 
$\mathbb R_0^q$ which satisfy the 
same product formula \eqref{prod_A}.
Now suppose that $\psi$ is a spherical function of $G/K$.
Then for $t\in \mathbb R^q,$ we have
\[ \psi(t) = \psi(t-\pi(t) + \pi(t)) = \,\psi(t-\pi(t))\cdot \psi(\pi(t))\]
because $\,t-\pi(t)$ corresponds to the scalar matrix $\, \exp\left( \sum_{i=1}^q t_i/q\right)\cdot I_q$
 which belongs to the subgroup $Z_{\mathbb R}:=\{a \cdot I_q:\> a>0\}$ of the center  of $G$.
As the  restriction of $\psi$ to $Z_{\mathbb R}$ is  multiplicative on  $Z_{\mathbb R}$, 
 we have $\psi (a\cdot I_q )= a^m$ with some exponent $m\in\mathbb C $.  Therefore
\begin{equation}\label{semisimple-auf-reduktiv} \psi(t) = \exp\bigl(m\cdot\sum_{i=1}^q t_i/q\bigr) \cdot 
\psi(\pi(t))
 \end{equation}
where the restriction $\psi|_{\mathbb R_0^q}$ corresponds to a spherical function of $G_1/K_1$. 
Conversely, it is easily checked that for a given spherical function $\psi$ of $G_1/K_1$,
formula \eqref{semisimple-auf-reduktiv} defines an extension to a   spherical function $\psi$ of $G/K$.

We now return to the usual convention and consider spherical functions as
functions on the group.
For $G_1/K_1$, the geometric multiplicity on the restricted root system
$\Delta=A_{q-1}$ is given by $m=d$. Therefore,
again according to Remark 2.3 of \cite{H3} and Lemma \ref{scaling}, 
the spherical functions of  $G_1/K_1$  can be identified as
\begin{equation}\label{spherical-a}
 \psi_\lambda(e^{\underline t}\,) \,=\, F_{A}(\lambda/2,d/2; 2t), \quad t\in \mathbb R_0^q,\end{equation}
with $ \lambda\in\mathbb C_0^q:=\{\lambda\in\mathbb C^q:\>
\sum_{i=1}^q\lambda_i=0\}$. For $\lambda \in \mathbb C^q, $ put $\, m = \sum_{i=1}^q \lambda_i\,.$ Then
\begin{equation}\label{def-m}\langle t-\pi(t),\lambda\rangle =m\cdot\sum_{i=1}^q t_i/q  \,.\end{equation}
This shows that we can parameterize the spherical functions of  $G/K$ according to
\begin{equation}\label{spherical-a-reduktiv}
 \psi_\lambda(e^{\underline t}) \,=\,e^{\langle t-\pi(t),\lambda\rangle} \cdot
 F_{A}\bigl(\pi(\lambda/2\bigr),d/2; \pi(2t)\bigr), \quad\lambda\in\mathbb C^q. \end{equation}
 With the notions of 
(\ref{grassmann-spherical-bc}) and (\ref{spherical-a-reduktiv}), Theorem
\ref{asymptot_cont} now implies the following limit relation.

\begin{corollary}\label{limtit-sphaerisch}
The spherical functions $\varphi_\lambda$ of  $\mathcal
G_{p,q}(\mathbb F)$ and
$ \psi_\lambda$ of $GL(q,\mathbb F)/U(q,\mathbb F)$ satisfy 
$$\lim_{p\to\infty}\varphi_{\lambda+\rho_{BC}^{geo}}(a_t)=
\psi_{\lambda+\rho_{A}^{geo}}({\rm cosh }\,\underline t\,)$$
for all $\lambda\in\mathbb C^q$ and $t\in\mathbb R^q$,  with  
the ``geometric'' constants $\rho_R^{geo} = 2\rho_{R}(k)$ given
by 
$$\rho_{BC}^{geo} = \sum_{i=1}^q (d(p+q+2-2i)-2)e_i  \quad\text{and}\quad 
\rho_A^{geo} = \sum_{i=1}^q d(q+1-2i)e_i .$$
\end{corollary}

\begin{proof} From relation (\ref{grassmann-spherical-bc}),  Theorem
\ref{asymptot_cont} and identity (\ref{def-m}) we obtain 
\begin{align}
\lim_{p\to\infty}&\varphi_{\lambda + \rho_{BC}^{geo}}(a_t)=\lim_{k_1\to\infty}
F_{BC}(\lambda/2 +\rho_{BC}(k), k;2t)\notag\\ 
&=  \prod_{i=1}^q ({\rm cosh}^2 t_i)^{\langle \lambda,\omega_q\rangle/2q}\cdot
F_A\bigl(\pi(\lambda/2) + \rho_A(k_3), d/2;
 \pi\bigl(\ln{\rm cosh}^2 t\bigr)\bigr)
\notag\\
&= e^{\langle \ln{\rm cosh}^2 t-\pi(\ln{\rm cosh}^2 t),
 \lambda/2\rangle}\cdot
F_A\bigl(\pi(\lambda/2)+\rho_A(k_3), d/2;
 \pi\bigl(\ln{\rm cosh}^2 t\bigr)\bigr), 
\notag\end{align}
with $k=(d(p-q)/2, (d-1)/2, d/2)$.
Using $\rho_A(k_3)\in\mathbb R_0^q$ and  (\ref{spherical-a-reduktiv}), we conclude that
this limit equals
\begin{multline*}
e^{\langle \ln{\rm cosh}^2 t-\pi(\ln{\rm cosh}^2 t),\,
 \lambda/2+\rho_A(k_3)\rangle}\cdot
F_A\bigl(\pi(\lambda/2+\rho_A(k_3)), d/2;
 \pi\bigl(\ln{\rm cosh}^2 t\bigr)\bigr)\\
=\,\psi_{\lambda+ \rho_{A}^{geo}}({\rm cosh}\, \underline t\,)\end{multline*}
as claimed.
\end{proof}

We finally mention that Corollary \ref{limtit-sphaerisch}  can be also obtained by sharp estimates
on the order of convergence, by comparing explicit versions of the
 Harish-Chandra integral representations 
of the involved spherical functions. This is  work in progress.
We also remark that our limit transition for hypergeometric functions has
a counterpart in the Euclidean case, namely the convergence
of (suitably scaled) Dunkl-Bessel functions of type B to such of type A, 
which was obtained in \cite{RV1} by completely different 
methods.

%%%%%%%%%%%%%%%%%%%%

\section{Spherical functions of infinite-dimensional Grassmannians}

We now discuss an interpretation of the preceding limit results in the context of infinite dimensional symmetric spaces
and Olshanski spherical pairs.
For the general background on this subject we refer to Faraut \cite{F} and
Olshanski \cite{Ol1}, \cite{Ol2} . In order
to be in agreement with standard terminology, we slightly change our notation.
We
consider the Grassmann manifolds $G_n/K_n$ with  $G_n=SO_0(n+q,q), \, SU(n+q,q) $ or $Sp(n+q,q)$ 
and maximal compact subgroup
 $K_n=SO(n+q)\times SO(q), \, S(U(n+q)\times U(q))$ or $Sp(n+q)\times Sp(q).$
 In all three cases, $G_n$
 is regarded as a closed subgroup of $G_{n+1}$ with $K_n=G_n\cap K_{n+1}$. Consider the inductive limits 
 $G_\infty = \lim_{\rightarrow} G_n$ 
and $K_\infty:= \lim_{\rightarrow} K_n$.  Then $(G_\infty,
 K_\infty)$ is an Olshanski spherical pair, and $G_\infty/K_\infty$ is 
one of the infinte-dimensional Grassmannians 
$SO_0(\infty,q)/SO(\infty)\times SO(q), \, SU(\infty,q)/S(U(\infty)\times U(q)), \, 
Sp(\infty,q)/Sp(\infty)\times Sp(q).$
A continuous function $\phi:G_\infty\to \mathbb C$  is called an Olshanski spherical 
function of $(G_\infty,K_\infty)$  if $\phi$ is $K_{\infty}$-biinvariant and satisfies the product
formula
$$\phi(g)\cdot \phi(h)=\lim_{n\to\infty}\int_{K_n} \phi(gkh)\>
dk\quad\quad\text{for}\quad g,h\in G_\infty.$$
We shall now classify the Olshanski spherical 
functions of $(G_\infty, K_\infty)$ without  representation theory.

For this we use the decomposition $G_n=K_n A_n^+ K_n$ 
\[   A_n^+:=\Biggl\{\begin{pmatrix}{\rm cosh}\,\underline t&0&{\rm sinh}\,\underline t \,\\ 0&I_n&0  \\
 {\rm sinh}\,\underline t&0&{\rm cosh}\,\underline t\
                            \end{pmatrix}:\>  t\in C_B\Biggr\}\]
of representatives of the $K_n$-double cosets in $G_n$, where again
$$C_B:=\{t=(t_1,\ldots,t_q)\in\mathbb R^q:\> t_1\ge t_2\ge\ldots\ge t_q\ge0\}$$
denotes the closed Weyl chamber of type $BC.$
Therefore, independently of $n$, we 
 identify $A_n^+$ with the set of diagonal matrices
\[ D :=\{{\rm cosh}\, \underline t:= {\rm diag}({\rm cosh }\, t_1,\ldots,{\rm cosh}\, t_q):\> t\in C_B\}.\]
This gives the
topological identification $G_n//K_n\simeq A_n^+ \simeq D.$
Notice that the elements of $D$ are just the lower right $q\times q$-blocks of
the matrices from $A_n^+$. 
In the same way,
\[G_\infty//K_\infty\simeq A_\infty^+ :=
\Biggl\{a_t^\infty:=\begin{pmatrix}{\rm cosh}\,\underline t&0&{\rm sinh}\,\underline t
\\ 0&I_\infty&0 
 \\ {\rm sinh}\,\underline t&0&{\rm cosh}\,\underline t\,
                            \end{pmatrix}:\>  t\in C_B\Biggr\}\simeq D.\]
By definition of the inductive limit topology, a function
$\phi:G_\infty\to\mathbb C$ is continuous and $K_\infty$-biinvariant iff for all
$n\in\mathbb N$, $\phi|_{G_n}$ is continuous and $K_n$-biinvariant. The space of all continuous,  
$K_\infty$-biinvariant functions on
$G_\infty$ may thus be identified with the space of all continuous functions on $D.$ Using this convention, the 
 Olshanski spherical 
functions of $(G_\infty, K_\infty)$ can be characterized as follows:

\begin{lemma}\label{limit-product-formula}
A continuous  $K_\infty$-biinvariant function $\phi:G_\infty\to \mathbb C$ is an Olshanski spherical function 
if and only if there is a continuous function  $\tilde\phi: D\to \mathbb C$ with
 $\phi(a_t^\infty)=\tilde\phi({\rm cosh}\, \underline t)$ for $t\in C_B$
such that  $\tilde\phi$ satisfies  the product
formula
\begin{equation}\label{prod-formel-1}
\tilde\phi(a)\cdot \tilde\phi(b)=
 \int_{U(q,\mathbb F)} \tilde\phi(\sigma_{sing}(akb)) dk, 
 \quad a, b\in D.
\end{equation}
Here the vector $\sigma_{sing}(...) \in \mathbb R^q$ is identified with the corresponding diagonal matrix. 
\end{lemma}

\begin{proof} Let $\phi$ be a continuous  $K_\infty$-biinvariant function on
  $G_\infty$. By the preceding discussion, $\phi$ is  Olshanski spherical if
  and only if there is a continuous function  $\tilde\phi:  D \to \mathbb C$ with 
  $\phi(a_t^\infty)=\tilde\phi({\rm cosh }\,t)$ for $t\in C_B$ such that  $\tilde\phi$ satisfies
\begin{equation}\label{limit-eq1}
\tilde\phi({\rm cosh }\,\underline t\,)\cdot \tilde\phi({\rm cosh }\,
\underline s\,)\,=\,\lim_{n\to\infty}\int_{K_n}\phi(a_t^\infty \, k \, a_s^\infty)\, dk 
=\lim_{n\to\infty}\int_{K_n}\phi(a_t^n \, k \, a_s^n)\, dk 
\end{equation}
for $s,t\in C_B$. We shall use Proposition 2.2 of \cite{R} to rewrite the integrals on
the right hand side. Let $B_q:=\{w\in M_q(\mathbb F):\> w^*w< I\}$ and
\[ c_n:=\int_{B_q}\Delta(I-w^*w)^{(n+q)d/2-\gamma} \>dw \quad \text{with }\, \gamma:=d(q-1/2)+1,\]
where $\Delta$ denotes the determinant and $dw$ means integration with respect to Lebesgue measure. Then
\begin{align}\label{prod-limes-formel}
\int_{K_n}\phi(a_t^n \, k \, a_s^n)\> dk 
=
 \,c_n^{-1}\int_{B_q}\int_{U_0(q,\mathbb  F)}&
\tilde\phi(\sigma_{sing}({\sinh}\, \underline t \> w \> {\rm sinh}\, \underline s\,+ \,{\rm cosh }\,\underline t \, k 
\,{\rm cosh }\, \underline s))\cdot \notag\\
&\cdot\Delta(I-w^*w)^{(n+q)d/2-\gamma} \> dk \> dw
\end{align}
where $U_0(q,\mathbb F)$ is the connected component of $U(q,\mathbb F).$
The probability measures
\[c_n^{-1}\cdot\Delta(I-w^*w)^{(n+q)d/2-\gamma} 
dw\]
are compactly supported in $B_q$ and tend weakly to the point
measure $\delta_0$ for $n\to \infty.$ Therefore (\ref{limit-eq1}) is equivalent to
 \begin{equation}\label{limit-eq2}
\tilde\phi({\rm cosh} \,\underline t\,)\cdot \tilde\phi({\rm cosh} \,\underline s\,)\,=\,
\int_{U_0(q,\mathbb
   F)}\tilde\phi(\sigma_{sing}({\rm cosh}\,\underline t \,k\,{\rm cosh} \,\underline s)) \> dk.
\end{equation}
Finally, is easily checked that the group $U_0(q,\mathbb F)$ 
may be replaced by $U(q,\mathbb F)$ in the integral, which completes the proof.
\end{proof}

We consider the reductive symmetric spaces $G/K= GL(q,\mathbb F)/U(q,\mathbb F)$ of subsection 
\ref{subsection-a-fall} and resume the notation from there. We introduce the set of diagonal matrices 
\[ D_0 := \{ e^{\underline t}\in M_q(\mathbb R): t= (t_1, \ldots t_q)\in \mathbb R^q \text{ with }\, t_1 \geq \ldots \geq t_q\}.\]
Then $G/\!/K \cong D_0$, and 
 a spherical function $\psi$ of $G/K$ may be characterized as a
continuous function  on $D_0$ satisfying the product formula
\begin{equation}\label{prod-formel-2}
\psi(a)\cdot\psi (b)=\,
\int_{U(q,\mathbb F)} \psi(\sigma_{sing}(akb))dk, \quad a,b \in D_0.
\end{equation}
Comparison with Lemma \ref{limit-product-formula} gives

\begin{theorem}\label{ident-ol}
A continuous  $K_\infty$-biinvariant function $\phi: G_\infty\to \mathbb C$ is an Olshanski spherical 
function if and only if the function  $\tilde\phi:  D\to \mathbb C$ with
 $\phi(a_t^\infty)=\tilde\phi({\rm cosh}\, t)$ for $t\in C_B$ is the restriction to $ D$ of a
spherical function $\psi$ of $G/K$. Each spherical function $\psi$ of $G/K$ is uniquely determined by its 
restriction to $ D$, and the Olshanski spherical functions therefore correspond in a bijective way to the spherical functions 
of $G/K.$ 
\end{theorem}

\begin{proof}
The if-part is clear from Lemma \ref{limit-product-formula}. The converse direction
follows from  Lemma \ref{limit-product-formula} together with the following lemma.
\end{proof}

\begin{lemma}
Each continuous function  $\varphi$ on $ D$ which satisfies  product formula (\ref{prod-formel-1})
admits
 a unique extension to a continuous function $\psi$ on $ D_0$ satisfying 
  product formula (\ref{prod-formel-2}).
\end{lemma}

\begin{proof}
 Assume first that  $\psi: D_0\to \mathbb C$ is such an extension of $\varphi.$  Consider first a 
scalar matrix $a=r I_q$ with $r\geq 1.$ Then $a\in D$
and hence $\psi(a)=\varphi(a).$ Moreover, as  $\psi(a^{-1})=1/\psi(a)$ for $a$ as above, the function
$\psi$ is uniquely determined by $\varphi$ on the set of scalar matrices $\,Z=\{r I_q, \,r >0\}.$  
Now let $a\in D_0$. We then find  $r >0$  and a matrix $b\in D$ 
 such that $a=r b.$ Using
 Product formula (\ref{prod-formel-2}) we obtain
\begin{equation}\label{prod-formel-3}
\psi(r I_q)\psi(b)\,= \,\int_{U_q(\mathbb F)} 
 \psi(\sigma_{sing}(r k b))\,dk\,=\, 
\psi(r b)=\psi(a).
\end{equation}
Therefore, $\psi$ is determined uniquely by $\varphi$. 

Conversely, it is easily checked that for  given $\varphi$, the definition
of 
$\psi$  first on $Z$ as above and then on $D_0$ via (\ref{prod-formel-3}) leads
to a well-defined continuous function $\psi$ on $D_0$ which satisfies the product formula.
\end{proof}

We notice at this point that our proof of Theorem \ref{ident-ol} relies only on the explicit product
  formula (\ref{prod-limes-formel})
 and does not require the results of the preceding sections. On the other hand,
  Corollary
\ref{limtit-sphaerisch} and Theorem \ref{ident-ol} imply the following

\begin{corollary}
All 
 Olshanski spherical functions of the infinite-dimensional Grassmannians $G_\infty/K_\infty$  
appear as limits of the spherical functions of the Grassmannians 
 $G_n/K_n$. 
\end{corollary}

Let us finally  remark that further
 Olshanski spherical pairs with fixed rank may be treated in  a similar way, for example  pairs related to the 
Cartan motion groups of  Grassmann  manifolds with growing dimension, see \cite{RV2}.

%%%%%%%%%%%%%%%%%%%%

\end{document}